\documentclass[12pt,leqno]{amsart}
\usepackage{amsmath,stmaryrd}
\usepackage{amssymb}
\usepackage{amsthm}
\usepackage{epsf}
\usepackage{graphicx}
\usepackage{esint} 
\usepackage{dsfont}
\usepackage[pagebackref]{hyperref} 
\hypersetup{pdfpagemode=FullScreen,  colorlinks=true} 
\usepackage{enumerate} 

\numberwithin{equation}{section}

\DeclareMathOperator{\Jac}{Jac}

\DeclareMathOperator{\diver}{div}

\DeclareMathOperator{\dist}{dist}

\def\div{\mathop{\operatorname{div}}}

\newcommand{\ms}{\medskip}

\newcommand{\R}{\mathbb R}

\newcommand{\bp}{\noindent {\em Proof: }}
\newcommand{\ep}{\hfill $\square$ \medskip}

\newcommand{\wt}{\widetilde}

\newcommand{\A}{\mathcal A}
\newcommand{\B}{\mathcal B}
\newcommand{\C}{\mathcal C}

\theoremstyle{plain}
\newtheorem{theorem}[equation]{Theorem}

\newtheorem{corollary}[equation]{Corollary}
\newtheorem{proposition}[equation]{Proposition}

\newtheorem{definition}[equation]{Definition}

\theoremstyle{definition}

\theoremstyle{remark}
\newtheorem{remark}[equation]{Remark}

\begin{document}

\title{A change of variable for Dahlberg-Kenig-Pipher operators.}

\author[Feneuil]{Joseph Feneuil}
\address{Joseph Feneuil. Universit\'e Paris-Saclay, Laboratoire de math\'ematiques d’Orsay, 91405, Orsay, France}
\email{joseph.feneuil@universite-paris-saclay.fr}
\thanks{This article was written during the author's stay at the Universit\'e Paris-Saclay in France, where he was supported by the Simons Foundation grant 601941, GD}

\maketitle

\begin{abstract} 
In the present article, we give a method to deal with Dahlberg-Kenig-Pipher (DPK) operators in boundary value problems on the upper half plane.
 
We give a nice subclass of the weak DKP operators that generates the full class of weak DKP operators under the action of bi-Lipschitz changes of variable on $\R^n_+$ that fix the boundary $\R^{n-1}$. Therefore, if one wants to prove a property on DKP operators which is stable by bi-Lipschitz transformations, one can directly assume that the operator belongs to the subclass. 
Our method gives an alternative proof to some past results and self-improves others beyond the existing literature.
\end{abstract}

\ms\noindent{\bf Key words:} Boundary value problems, Elliptic operators with rough coefficients, Dahlberg-Kenig-Pipher operators, Carleson perturbations.

\ms\noindent
AMS classification:  35J25.

\tableofcontents

\section{Introduction}
\label{S1}

In the Euclidean setting, boundary value problems for elliptic operators have been studied for ages by many researchers, who worked to find the relationship between the smoothness of the boundary data, the geometry of the boundary, and the smoothness of the coefficients of the elliptic operator under consideration.

Let us skip the early literature on the cases where the boundary datum is continuous or where the ambient space is the complex plane. In the higher dimensional Euclidean setting, the Dirichlet value problem in $L^p$ was studied for the Laplacian and rough boundaries in for instance \cite{Dahlberg77}, \cite{BJ}, \cite{DJ}, \cite{KenigB}, \cite{HM14}, \cite{HMUT14}, \cite{AHM3TV}, \cite{AHMNT}, \cite{AHMMT}, \cite{Azz}. When the operator is not the Laplacian, negative results on the solvability of the Dirichlet problem in $L^p$ are in \cite{CFK}, \cite{MM}, and positive results are for either $t$-independent operators on $\R^{n}_+:= \{(x,t)\in \R^{n-1} \times (0,\infty)\}$ (see for instance \cite{JK}, \cite{KKPT}, \cite{HKMP15}), or for operators that verify some conditions described in terms of Carleson measures (see \cite{FKP}, \cite{KP93}, \cite{KePiDrift}, \cite{DPP07}, \cite{DP}). Of course, it would be impossible to mention all the articles that could be linked to the topic, in particular knowing that the present article is meant to be brief, hence the aforementioned papers are only a small sample of the considerable work in the research area. We refer to the introduction of \cite{DFMKenig} and \cite{FP} for a better description of the current state of the art.

The operators that fall into the scope of the present article are in the second form, i.e. related to Carleson measures. The Dahlberg-Kenig-Pipher operators - or DKP operators for short - are roughly the uniformly elliptic operators in the form $L=-\diver \B \nabla$ on a domain $\Omega$ for which
 \begin{equation} \label{defDKP}
  \dist(X,\partial \Omega) |\nabla \B|^2 dX \text{ is a Carleson measure on $\Omega$.}
 \end{equation}
More precise definitions will be given later. For the DKP operators (and sufficiently smooth boundaries), the Dirichlet problem is solvable in $L^p$ for sufficiently large $p\in (1,\infty)$. This assertion was claimed by Dahlberg and proved by Kenig and Pipher in \cite{KePiDrift}, hence the name `DKP'. Kenig's and Pipher's proof is based on the following simple yet powerful observation. Consider a matrix $\B$ such that $\nabla \B$ is at least in $L^1_{loc}$. We write $\B$ as a matrix by blocs
\begin{equation} \label{defBi}
\B := \begin{pmatrix} \B_1 & \B_2 \\ \B_3 & b \end{pmatrix},
\end{equation}
where $b$ is a scalar function on $\Omega$, and so $\B_1$ is a matrix function of order $n-1$, and $\B_2$ and $\B_3$ are respectively vertical and horizontal vector functions of length $n-1$. We construct $\wt \B$ and $\mathcal T$ as
\[\wt \B := \begin{pmatrix} b^{-1} \B_1 & b^{-1}(\B_2+ \B_3^T) \\ 0 & 1 \end{pmatrix} \quad \text{ and } \quad \mathcal T := \begin{pmatrix} 0 & \B_3^T \\ -\B_3 & 0 \end{pmatrix},\]
and we notice that we have the formal computation
\begin{equation} \label{BtoBD} \begin{split}
-\div \B \nabla & = - b \div\Big( \frac\B b \nabla \Big) - \frac{\nabla b}{b} \cdot \B \nabla \\
& = - b \div\Big( \frac{\B + \mathcal T}{b} \nabla \Big) + b \div\Big(\frac{\mathcal T}{b} \nabla \Big) - \frac{\nabla b}{b} \cdot \B \nabla \\
& = - b \div\Big( \wt \B \, \nabla \Big) + b \underbrace{\left[ \div\Big(\frac{\mathcal T}{b}\Big) - \B^T \frac{\nabla b}{b^2} \right]}_{\mathcal D} \cdot \nabla, 
\end{split} \end{equation}
where the divergence of a matrix $\A$ is a vector containing the divergence of the columns of $\A$. If $\B$ is uniformly elliptic - see \eqref{defellip} below - and if $u$ is a weak solution to $-\div \B \nabla u = 0$, then $u$ is also a weak solution to the elliptic operator with drift  $(-\div \wt \B \nabla + \mathcal D \cdot \nabla )u = 0$.
Following \cite{KePiDrift}, this transformation has been used by numerous authors  to turn an elliptic operator into another elliptic operator with (for instance) a nice last line, at the price of adding a drift (see for instance \cite{DP}, \cite{DPP07}, \cite{DPR}, \cite{FMZ}). 
The same trick can be used to get 0 in the top right corner of $\B$ (instead of the bottom right corner) or to turn $\B$ into a symmetric matrix, always at the price of adding a drift.
The identity \eqref{BtoBD} is so easy that it may not even be discussed by the authors that used it, and may often be only a couple of lines in the deepest part of the proofs. But we did not know any way to avoid it, that is, before the present article, we were not aware of any method that allows people to avoid drifts when dealing with DKP operators. The apparently innocent observation \eqref{BtoBD}, which allows us to trivialize one line (or column) of the matrix $\B$, has been crucial in many proofs. However, it is obviously limited to cases where we can handle drifts.

\medskip

In this paper, we purpose an alternative method to the transformation \eqref{BtoBD}, which similarly allows us to assume that one line of $\B$ is simple, and the price to pay this time is a ``Carleson perturbation''. Carleson perturbations are a class of perturbations on the matrix $\A$ of a uniformly elliptic operator $L=\div \A \nabla$ which will preserve the solvability of the Dirichlet problem in $L^p$ for large $p$ (\cite{FKP}), and they are even optimal in some sense (see \cite{FKP} again). Morally speaking, Carleson perturbations are objects that fit very well in the theory of boundary value problems with data in $L^p$, and many results in the recent literature include them.
So we are saying that we purpose here an alternative method that trivializes a line of $\B$ by adding a term which is often free (because already included in the theory, because natural in the theory).
Of course, we do not say that the argument presented below is strictly better than using \eqref{BtoBD}. Both have pros and cons that we shall discuss later. Our aim is really to introduce the reader to a substitute method, and to show that this other method can be used to self improve some existing results in situations where using \eqref{BtoBD} does not work. 

\medskip

We shall present our method in a fairly simple case. Yet, it is meant to be adapted to more complicated scenarios, some of them will be discussed, but of course we cannot anticipate all of them. Our main choice towards simplification is that our domain will always be the upper half plane $\R^{n}_+$. In this context, we define the Carleson measures and the Carleson measure condition as follows.

\begin{definition}
We say that a quantity $f$ (a scalar, vector, or matrix function) defined on $\R^n_+$ satisfies the Carleson measure condition if $|f|^2 \frac{dt}{t}dx$ is a Carleson measure, that is if for any $x\in \R^{d}$ and any $r>0$, 
\begin{equation} \label{defCM}
\int_{B(x,r)} \int_{0}^r |f(y,t)|^2 \, \frac{dt}{t} \, dy \leq Mr^{d}.
\end{equation}
We write $f\in CM$ - or $f\in CM(M)$ when we want to refer to the constant in \eqref{defCM} - when $f$ satisfies the Carleson measure condition.

In addition, we write that $f\in CM_{sup}$ if $f_{sup} \in CM$, where $f_{sup}$ is the function on $\R^{d+1}_+$ defined as
\[f_{sup}(y,t):=\sup\{|f(z,s)|, \, s\in (t/2,2t),\, z\in B(y,t)\}.\]
Note that $f\in CM_{sup}$ is equivalent to $f_{sup} \in CM_{sup}$, and implies that $f\in CM$ and $f\in L^\infty(\R^n_+)$.
\end{definition}

\medskip

Let us turn to the operators under consideration. An operator $L:=-\diver \A \nabla$ is said to be uniformly elliptic if there exists $C_A$ such that 
\begin{equation} \label{defellip}
\A(x,t) \xi \cdot \xi \geq (C_A)^{-1} |\xi|^2 \text{ and } |\A(x,t) \xi \cdot \zeta| \leq C_A |\xi||\zeta| \qquad \text{ for } (x,t) \in R^n_+, \, \xi,\zeta \in \R^n.
\end{equation}

The Dahlberg-Kenig-Pipher operators are the uniformly elliptic operators on $\R^n_+$ for which $|t\nabla \A| \in CM_{sup}$ or $|t\nabla \A| \in CM \cap L^\infty$, depending on the context. Recall that the solvability of the Dirichlet problem in $L^p$ is stable under Carleson perturbation on the coefficients (\cite{FKP}), therefore it means that a larger class of DKP-type operators on $\R^n_+$ can be given by assuming that the matrix $\A$ can be decomposed as $\A = \B + \C$ with $|t\nabla \B| + |\C| \in CM_{sup}$.

\bigskip

Our main result is given as follows.

\begin{proposition} \label{Main2} Let $L:=-\div \A \nabla$ be a uniformly elliptic operator on $\R^n_+$ such that $\A$ can be decomposed as $\A = \B + \C$ where $|t\nabla \B| + |\C| \in CM_{sup}(M)$. Then there exists a bi-Lipschitz change of variable $\rho$ that maps $\R^n_+$ to itself and fixes the boundary $\R^{n-1}$ such that the conjugate operator $L_\rho:= -\div \A_\rho \nabla$ verifies that $\A_\rho$ can be decomposed as $\A_\rho =\B_\rho + \C_\rho$ where $\B_\rho$ has the form
\begin{equation}
\B_\rho = \begin{pmatrix} \B_{1,\rho} & \B_{2,\rho} \\ 0 & 1 \end{pmatrix}
\end{equation}
and $t|\nabla \B_\rho|+ |\C_\rho| \in CM_{sup}(M')$. We can give a bound on $M'/M$ and on the Lipschitz constants in $\rho$ and  that depends only on $n$ and the ellipticity constant $C_A$.
\end{proposition}

In the proposition, by {\em conjugate operator}, we mean the operator $L_\rho$ for which $v := u\circ \rho$ is a weak solution to $L_\rho v = 0$ whenever $u$ is a weak solution to $Lu=0$. 

\medskip

\noindent{ \em Remarks:}
\begin{enumerate}[(i)]
\item We insist on the fact that the aim of our paper is to present the method, and not to give the most general result possible. In future articles, we intend to adapt the argument in other contexts by explaining the differences with the proof given here (instead of repeating all the details).
\item The precise type of conditions on $\B$ and $\C$ do not really matter, as long as they are stable by bi-Lipschitz transformation. For instance, assuming $|t\nabla \B| + |\C| \in CM$ (instead of $CM_{sup}$) would give the analogous conclusion where $t|\nabla \B_\rho|+ |\C_\rho| \in CM$ (instead of $CM_{sup}$). 
\item We even only need a Carleson condition on the last line of $\B$ to get the simplification of the last line of $\B_\rho$ (but if we only assume the Carleson condition on the last line of $\B$, we cannot conclude anymore that $|t||\nabla \B_\rho| \in CM_{sup}$).
\item With the same technique, we could find $\rho$ so that $\B_\rho$ is zero in the top right corner (instead of the bottom left one). And if $\B$ is symmetric, then $\B_\rho$ would be diagonal by blocs.
\item The proof can immediately be adapted  to operators on $\R^n \setminus \R^d := \{(x,t)\in \R^d \times \R^{n-d}, \, t\neq 0\}$ ($d<n-1$) in the form $L=-\diver [|t|^{d+1-n} \A \nabla]$ with $\A$ as in \eqref{defellip}. Here $\R^n\setminus \R^d$ is  the analogue of $\R^n_+$ in domains with higher codimensional boundaries. The elliptic theory associated to such operators was studied in \cite{DFMprelim} (see also \cite{FKS}) and the solvability of the Dirichlet problem was studied under various situations in for instance \cite{DFMAinfty}, \cite{FMZ}, \cite{FenAinfty}, \cite{DM}, \cite{MP}, \cite{FenGinfty}. 
We really hesitated to include this case in the proposition, but we decided that it was not worth complicating the argument further. A quick discussion about the differences is given in the proof of Theorem \ref{DpTh} below. 
\item For elliptic operators with complex coefficients (see \cite{DP}, \cite{FMZ}), our method allows us to remove the real part of the last line of $\B$, but not the imaginary part. Yet using \eqref{BtoBD} yields the same reduction and the same limitation.
\item Unfortunately, contrary to the transformation \eqref{BtoBD}, our method does not seem easily adaptable to systems of elliptic equations. 
\item On the other hand, Proposition \ref{Main2} is well adapted to the regularity problem and the Neumann problem, where a theory on Carleson perturbations exists (see \cite{KP95}). We give an application of Proposition \ref{Main2} in Corollary \ref{RpCor} below.
\item The method is dependent on the invariance of the domain by tangential translations and by dilatations. However, there are many ways to reduce the domain to $\R^n_+$. For instance, in a paper under preparation \cite{FLM}, we will use a variant of the method to prove a result on uniformly rectifiable domains.
\item To the best of the author's knowledge, the only time where a change of variable has been used to simplify the coefficients of an elliptic operator is in \cite{HKMP15}. But the context of \cite{HKMP15} ($t$-independent operators) is different from our context. 
\end{enumerate}

In the rest of the article, we prove Proposition \ref{Main2}, and we give two examples where the method extends the results from the existing literature (Corollary \ref{RpCor} and Theorem \ref{DpTh}). 

The notation $A\lesssim B$ is employed when $A \leq C B$ for a constant that depends only on parameters that will be either recalled or obvious from context. We shall also use $A\approx B$ when $A \lesssim B$ and $A \gtrsim B$.

\section{Proof of Proposition \ref{Main2}}

Let $\A$, $\B$ and $\C$ be like in the assumptions of the proposition, that is $\A$ is uniformly elliptic with a constant $C_A$,  $\A := \B + \C$, and $|t\nabla \B| + |\C| \in CM_{sup}(M)$.

\medskip

\noindent {\bf Step 1: Construction of $\rho_{v,h}$.}  For a positive scalar function $h$ and a vector function $v$, we define
\begin{equation} \label{defrho12}
\rho(y,t) = \rho_{v,h} (y,t) := (y + tv(y,t), h(y,t)t), 
\end{equation}
which fixes the boundary $\R^{n-1}$ and which is a Lipschitz map whenever $v$, $h$, $|t\nabla v|$, and $|t\nabla h|$ are uniformly bounded. Its Jacobian is
\begin{equation} \label{defJacrho}
Jac_\rho = \begin{pmatrix} I + t\nabla_x v  & t\nabla_x h \\ v + t\partial_t v & h +t\partial_t h \end{pmatrix}.
\end{equation}
Let $J_\rho$ be the matrix function
\begin{equation} \label{defJrho}
J_\rho := \begin{pmatrix} I & 0 \\ v & h \end{pmatrix}.
\end{equation}
Since $|Jac_\rho - J_\rho| \lesssim |t\nabla h| + |t\nabla v|$, we can find $\epsilon_0>0$ small enough and $C_0$ large enough (both depending on $d$, $n$ and $\|h^{-1}\|_\infty$) so that, if 
\begin{equation} \label{condh}
\|t\nabla h\|_\infty +  \|t\nabla v\|_\infty < \epsilon_0,
\end{equation}
then for any  $(y,t) \in \Omega_0$,  $Jac_\rho(y,t)$ is invertible,
\begin{equation} \label{condh2}
 |\det(Jac_\rho(y,t))-h(y,t)| \leq (2\|h^{-1}\|_\infty)^{-1}, 
\end{equation}
and
\begin{equation} \label{condh3}
 |(Jac_\rho)^{-1}(y,t) - (J_\rho)^{-1}(y,t)| \leq C_0 (|t\nabla h(y,t)| + |t\nabla v(y,t)|).
\end{equation}

If $h$ and $v$ satisfy \eqref{condh}, which we now assume, then $\rho$ is a bi-Lipschitz change of variable. So we can look for the conjugate of the operator $L=-\diver \A \nabla$ by $\rho$. We let the reader check that if $u$ is a weak solution to $L=-\diver \mathcal A \nabla$, then $u\circ \rho$ is solution to $L_{\rho} = -\diver \A_\rho \nabla$ where 
\begin{equation}
\A_{\rho} = \det(Jac_\rho) (Jac_\rho)^{-T} (\mathcal A \circ \rho) \, (Jac_\rho)^{-1}.
\end{equation}
To lighten the notation, we use $\mathcal O(f)$ to denote any scalar/vector/matrix function on $\Omega$ which can be bounded by $Cf$ (where $C$ depends only on $d$, $n$, $\|h^{-1}\|_\infty$ and later $C_A$). In particular, the quantity $\mathcal O(f)$ can change from one line to another. 
We use this notation to rewrite \eqref{condh3} as
\[(Jac_\rho)^{-1} = (J_\rho)^{-1} + \mathcal O(|t\nabla h| + |t\nabla v|) \quad \text{ and } \quad (Jac_\rho)^{-T} = (J_\rho)^{-T} + \mathcal O(|t\nabla h| + |t\nabla v|) .\]
Of course, we also have 
\[\det(Jac_\rho) =  h + \mathcal O(|t\nabla h| + |t\nabla v|).\]
Remember that the matrix $\mathcal A$ equals $\A = \B + \C = \B + \mathcal O(|\C|)$, so if we write $\B$ as the matrix by blocs
\begin{equation} \label{defBi} 
\B = \begin{pmatrix} B_1 & B_2 \\ B_3 & b \end{pmatrix},
\end{equation}
then
\begin{equation} \label{Arho}
\begin{split}
\A_{\rho}
& = h \begin{pmatrix} I & -h^{-1}v^T \\ 0 & h^{-1} \end{pmatrix} \begin{pmatrix} B_1 \circ \rho & B_2 \circ \rho \\ B_3 \circ \rho & b \circ \rho \end{pmatrix} \begin{pmatrix} I & 0 \\ -h^{-1}v & h^{-1} \end{pmatrix} + \mathcal O(|\C\circ \rho| + |t\nabla h| + |t\nabla v|)\\
& =  \begin{pmatrix} * & B_2\circ \rho - (b\circ \rho) h^{-1}v^T \\ B_3\circ \rho - (b\circ \rho) h^{-1}v & (b\circ \rho) h^{-1} \end{pmatrix} + \mathcal O(|\C \circ \rho| + |t\nabla h| + |t\nabla v|) \\
& = \underbrace{\begin{pmatrix} * & B_2 - b h^{-1}v^T \\ B_3 - b h^{-1}v & bh^{-1} \end{pmatrix}}_{\B_{\rho}}  + \mathcal O(|\B\circ \rho - \B| + |\C \circ \rho| + |t\nabla h| + |t\nabla v|),
\end{split}\end{equation}
where the top left corner of $\B$ is not given to lighten the computations, but depends only on products, quotients, differences, and sums of coefficients of $\B$, $h$, and $v$. We write $\C_{\rho}$ for $\A_{\rho} - \B_{\rho}$, and we have that
\begin{equation} \label{defCrho}
 |\C_\rho| \lesssim |\B\circ \rho - \B| + |\C\circ \rho| + |t\nabla h| + |t\nabla v|. 
\end{equation}

\medskip

\noindent {\bf Step 2:} In the change of variable $\rho_{v,h}$ constructed in Step 1, we really want to take $h:=b$ and $v:=B_3$. However, with those choices, nothing guarantees that \eqref{condh} is satisfied ... and we may not have a change of variable.
Still, we can see that it works when $t\nabla b$ and $t\nabla B_3$ are small enough. So the strategy will be to do multiple changes of variables, each of them canceling only a small portion of $b$ and $B_3$, but altogether reducing $b$ to 1 and $B_3$ to 0.

For this, we want to do an induction on the size of $t\nabla b$ and $t\nabla B_3$. The issue is that we do not have a control on those terms, hence no control on the number of change of variable that we need (so ultimately no control on the Carleson constant $M'$ in the conclusion of the Proposition). That is a bit annoying, so we need the following observation: without loss of generality, we can assume that 
\begin{equation} \label{Bellip}
\text{$\B$ satisfies \eqref{defellip} with the constant $C_A$}
\end{equation}
and
\begin{equation} \label{tNB<CA}
\|t\nabla \B\|_\infty \leq C C_A,
\end{equation}
where $C>0$ depends only on the dimension (via our choice for $\theta$ below). 
To prove this, we shall construct $\wt \B$ which satisfies \eqref{Bellip}, \eqref{tNB<CA}, and $|t\nabla \wt \B| \in CM_{sup}(CM)$ such that $\wt \C:= \A-\wt \B \in CM_{sup}(CM)$.

We choose once for all the proof a bump function $\theta \in C^\infty_0(\R^n)$ supported in $B(0,1/2)$, that is $0 \leq \theta \leq 1$ and $\iint_{\R^n} \theta \, dX = 1$. We construct $\theta_{y,t}(z,s) = t^{-n}\theta\big(\frac{z-y}{t},\frac{s-t}{t}\big)$, which satisfies $\iint_{\R^n} \theta_{y,t} = 1$. Note, if $\Theta(X) := X\theta(X)$, that 
\begin{equation} \label{difftheta}
t\nabla_{y,t} \theta_{y,t}(z,s) = - t^{-n} \nabla \theta\Big( \frac{z-y}{t},\frac{s-t}{t} \Big) - t^{-n} \div(\Theta)\Big( \frac{z-y}{t},\frac{s-t}{t} \Big),
\end{equation}
in particular $\|t\nabla_{y,t} \theta_{y,t}\|_{L^1} \leq C_\theta$.
We define
\begin{equation} \label{defwtB}
\wt{\B}(y,t) := \iint_{\R^n} \A(z,s)  \, \theta_{y,t}(z,s) \, dz\, ds.
\end{equation}
Since $\wt{\mathcal B}$ is an average of $\A$, then having \eqref{Bellip} for $\wt \B$ is immediate. 
Observe that
\begin{equation}
|t \nabla \wt{\B}(y,t)|  \leq \||\A|\|_\infty \iint_{\R^n} |t\nabla_{y,t}\theta_{y,t}(z,s)| \, dz\, ds \lesssim C_A.
\end{equation}
Moreover, using that $\A = \B + \C$ in \eqref{defwtB}, we have by \eqref{difftheta} that
\[ \begin{split}
|t \nabla \wt{\B}(y,t)| & \leq t^{-n} \left| \iint_{\R^n} [\B(z,s) + \C(z,s)] \left[ \nabla \theta\Big( \frac{z-y}{t},\frac{s-t}{t} \Big) + \div(\Theta)\Big( \frac{z-y}{t},\frac{s-t}{t} \Big) \right] \, dz\, ds  \right| \\
& \leq t^{-n}\iint_{\R^n} |t\nabla \B(z,s)| \left|\theta\Big( \frac{z-y}{t},\frac{s-t}{t} \Big)  \right| \, dz\, ds   \\
& \qquad + t^{-n} \iint_{\R^n} |\C(z,s)| \left| \nabla \theta\Big( \frac{z-y}{t},\frac{s-t}{t} \Big) + \div(\Theta)\Big( \frac{z-y}{t},\frac{s-t}{t} \Big) \right| \, dz\, ds \\ 
& \lesssim \sup_{\begin{subarray}{c} z\in B(y,t) \\ t/2<s<2t \end{subarray}} |t\nabla \B(z,s)| + |\C(z,s)| 
\end{split} \]
where, for the second inequality, we used an integration by part to move the divergence and the gradient from $\theta$ and $\Theta$ to $\B$. We conclude that $t\nabla \wt\B \in CM_{sup}$ as desired.
It remains to verify that $\wt \C \in CM_{sup}$. Indeed,
\[ \begin{split}
|\wt \C(y,s)| & = |\A(y,t) - \wt\B(y,t) | = \iint_{\R^n} \Big[|\B(y,t)-\B(z,s)| + |\C(y,t)| + |\C(z,t)| \Big] \, \theta_{y,t}(z,s) \, dz\, ds  \\
& \lesssim \sup_{\begin{subarray}{c} z\in B(y,t) \\ t/2<s<2t \end{subarray}} |\B(y,t)-\B(z,s)| + |\C(z,s)| \lesssim \sup_{\begin{subarray}{c} z\in B(y,t) \\ t/2<s<2t \end{subarray}} |t\nabla \B(z,s)| + |\C(z,s)| \in CM_{sup}.
\end{split} \]
Step 2 follows.

\medskip

\noindent {\bf Step 3:} We write $\B$ as the matrix by blocs
\begin{equation} \label{defBi} 
\B = \begin{pmatrix} B_1 & B_2 \\ B_3 & b \end{pmatrix},
\end{equation}
where $b$ is a scalar function, so $B_1$ is a matrix of order $n-1$, $B_2$ and $B_3$ are respectively a vertical and a horizontal vector of length $n-1$. According to Step 2, we can assume that $\B$ satisfies \eqref{Bellip} and \eqref{tNB<CA}, so
\begin{equation} \label{tNb<CA}
\|b\|_\infty + \|b^{-1}\|_\infty + \|\B_3\|_\infty + \|t\nabla b\|_\infty + \|t\nabla B_3\|_\infty \lesssim 1,
\end{equation}
where the constants depends only on $n$ and $C_A$. We want to use the change of variable \eqref{defrho12} with $h=b^{1/N}$ and $v_k = \frac1N b^{-k/N}B_3$ ($0 \leq k \leq N$) for a large enough $N$. So we need to verify \eqref{condh}. We have
\[\begin{split}
\|t \nabla h\|_\infty \lesssim \frac1N \|b^{-1}\|_\infty \|t\nabla b\|_\infty \lesssim \frac1N
\end{split}\]
by \eqref{tNb<CA}. Moreover, similar computations yield, for $0 \leq k \leq N$,
\[\begin{split}
\|t \nabla v_k\|_\infty & \lesssim \frac1N (\|b^{-2}\|_\infty \|B_3\|_\infty \|t\nabla b\|_\infty + \|b^{-1}\|_\infty \|t\nabla B_3\|_\infty)  \lesssim  \frac1N.
\end{split}\]
Altogether, we have for any $0 \leq k \leq N$
\begin{equation} \label{condhZ}
\|t \nabla h\|_\infty  + \|t \nabla v_k\|_\infty \leq \frac{C(C_A,n)}N < \epsilon_0
\end{equation}
for some large enough $N$ that depends only on $C_A$ and $n$, that we fix for the rest of the proof. By possibly taking $N$ slightly bigger (still depending only on $C_A$ and $n$), we can even assume that $N$ is such that
\begin{equation} \label{condNb}
\frac12 < h < 2 \quad \text{ and } \quad \sup_{0\leq k \leq N} |v_k| \leq 1
\end{equation}
to simplify the incoming computations.

\medskip

By \eqref{condhZ}, the map $\rho_{N-1}$ constructed in \eqref{defrho12} with of $h$ and $v_{N-1}$ is a bi-Lipschitz change of variable. Therefore, thanks to \eqref{Arho} and \eqref{defCrho}, the conjugate operator $L_{N-1} = -\diver \A_{N-1} \nabla$ of $L$ by $\rho_{N-1}$ is such that $\A_{N-1}$ can be decomposed as
\[\A_{N-1}:= \underbrace{\begin{pmatrix} * & * \\ \frac{N-1}{N} B_3 & b^{(N-1)/N} \end{pmatrix}}_{\B_{N-1}} + \ \C_{N-1}\] 
where $|t\nabla \B_{N-1}| \in CM_{sup}$ (because $\B_{N-1}$ is the product, quotient, difference or sum of coefficients of $\B$) and $\C_{N-1} \in CM_{sup}$ as well because
\[\begin{split}
|\C_{N-1}| \lesssim |\B\circ \rho - \B| + |\C\circ \rho| + |t\nabla h| + |t\nabla v|
\end{split}\]
by \eqref{defCrho},
\[|t\nabla h| + |t\nabla v| \lesssim |t\nabla \B| \in CM_{sup},\]
and since $\rho(y,t) \in B(y,t) \times (t/2,2t)$ by \eqref{condNb},
\[|\C\circ \rho(y,t)| + |\B\circ \rho(y,t) - \B(y,t)| \lesssim \sup_{\begin{subarray}{c} |z-y|< t \\ t/2 < s < 2t \end{subarray}} \Big[ |\C(z,s)| + |t\nabla \B(z,s)|  \Big] \in CM_{sup}.\]
We iterate the process, and so the conjugate operator $L_{k-1}$ of $L_k$ by $\rho_{k-1}$ (constructed with $h$ and $v_{k-1}$) is $-\diver \A_{k-1} \nabla$ where $\A_{k-1}$ can be decomposed as
$\A_{k-1} = \B_{k-1} + \C_{k-1}$ where
\[\B_{k-1} = \begin{pmatrix} * & * \\ \frac{k-1}{N} B_3 & b^{(k-1)/N} \end{pmatrix}, \]
and $|\mathcal C_{k-1}| + |t\nabla \B_{k-1}| \in CM_{sup}$. The proposition follows when we reach $k=0$.

\section{Applications}

\subsection{Regularity problem}

The first applications that we can think of are related to the Dirichlet and regularity problems, that we shall now introduce properly. 

Given a uniformly elliptic operator $L$ on $\R^n_+$, we can construct an associated elliptic measure $\{\omega^X_L\}_{X\in \R^n_+}$ on the boundary $\R^{n-1}$. That is, for any $g\in C_c(\R^{n-1})$ - the set of continuous and compactly supported functions on $\R^{n-1}$ - the function $u_g$ defined for all $X\in \R^n_+$ as 
\begin{equation} \label{defug}
u_g(X) := \int_{\R^{n-1}} g(y) \, d\omega^X_L(y) \quad \text{ for } X\in \R^n_+
\end{equation}
is the (unique) solution to $Lu = 0$ which is bounded, continuous up to the boundary, and satisfies $u=g$ on the boundary $\R^{n-1}$. We can use the construction \eqref{defug} when $g \in L^\infty(\R^{n-1})$ or even $BMO(\R^{n-1})$. In those cases $u_g$ will still be a solution to $Lu=0$ but the fact that $u=g$ on $\R^{n-1}$ has to be taken in a weaker sense (for instance non-tangential limit almost everywhere). But we  cannot guarantee that \eqref{defug} makes sense if $g$ lies only in $L^p$ for $p<\infty$. This issue is called the Dirichlet problem.

\begin{definition}[Dirichlet problem] \label{defDp}
We say that the Dirichlet problem for $L$ is solvable in $L^p$ if there exists a constant $C>0$ such that for any $g\in C_c(\R^{n-1})$, the solution $u_g$ constructed in \eqref{defug} satisfies
\begin{equation} \label{Nu<g}
\|N(u_g)\|_p \leq C \|g\|_p,
\end{equation}
where $N$ is the non-tangentially maximal function defined as 
\begin{equation} \label{defN}
N(u)(x) := \sup_{(y,t)\in \Gamma(x)}  u
\end{equation}
and $\Gamma(x):= \{(y,t) \in \R^n_+, \, |x-y|< t\}$ is the vertical cone with vertex in $x$ and aperture $1$.
\end{definition}

When the Dirichlet problem is solvable in the sense given in the above definition, then we can use the density of $C_c(\R^{n-1})$ in $L^p$ and \eqref{Nu<g} to say that a solution can be constructed using \eqref{defug} for any $g\in L^p(\R^{n-1})$, hence proving the existence of solutions for any data in $L^p$. 

The regularity problem is similar to the Dirichlet problem, but on the gradients. 

\begin{definition}[Regularity problem] \label{defRq}
We say that the regularity problem for $L$ is solvable in $L^q$ if there exists a constant $C>0$ such that for any $g\in C_c(\R^{n-1})$, the solution $u_g$ constructed in \eqref{defug} satisfies
\begin{equation} \label{NNu<Ng}
\|\wt N(\nabla u_g)\|_q \leq C \|\nabla g\|_q,
\end{equation}
where $\wt N$ is the averaged non-tangentially maximal function defined as 
\begin{equation} \label{defwtN}
\wt N(\nabla u)(x) := \sup_{(y,t)\in \Gamma(x)}  \left(\fint_{z\in B(y,t)} \fint_{t/2<s<2t} |\nabla u|^2 \, ds\, dz\right)^{\frac12}.
\end{equation}
Such averaged version of $N$ is needed here because, contrary to $u_g$, the quantity $\nabla u_g$ is not necessarily locally bounded.
\end{definition}

A first result on the regularity problem, which is key to our argument, is:

\begin{proposition} \label{Prrhostab}
The regularity problem in $L^q$ is stable under a bi-Lipschitz change of variable $\rho: \, \R^n_+ \to \R^{n}_+$ that fixes the boundary $\partial \R^n_+ = \R^{n-1}$. 

In addition, the ratio between the constants in \eqref{NNu<Ng} before and after the change of variable depends only on the dimension $n$ and the bi-Lipschitz constants of $\rho$. 
\end{proposition}

\begin{proof}
Let $L$ be a uniformly elliptic operator and $L_\rho$ be its conjugate by $\rho$. By uniqueness of the harmonic measure, we have $\omega_{L_\rho}^{X} = \omega_L^{\rho(X)}$, thus the solutions $u_g$ and $v_{g}$ to respectively $Lu =0$ and $L_\rho v = 0$ constructed by \eqref{defug} verify $v_g = u_g \circ \rho$.

From there, we only need to show that $\|\wt N(\nabla u)\|_q \approx \|\wt N(\nabla [u\circ \rho])\|_q$, and even
\begin{equation} \label{claimlast}
\|\wt N(\nabla [u\circ \rho])\|_q \lesssim \|\wt N(\nabla u)\|_q .
\end{equation}
because the proof of other bound is identical by using $\rho^{-1}$ instead of $\rho$. 

Let $x\in \R^{n-1} = \partial \R^n_+$, and take $(y,t) \in \Gamma(x)$. The map $\rho$ is bi-Lipschitz, so the Jacobian matrix $\Jac_\rho$ and $1/\det(\Jac_\rho)$ are uniformly bounded. Therefore, 
\begin{multline*}
\fint_{z\in B(y,t)} \fint_{t/2<s<2t} |\nabla[u \circ \rho]|^2 \, ds\, dz  = \fint_{z\in B(y,t)} \fint_{t/2<s<2t} |\Jac_\rho (\nabla u) \circ \rho|^2 \, ds\, dz \\
 \lesssim \fint_{z\in B(y,t)} \fint_{t/2<s<2t} |(\nabla u) \circ \rho|^2 \, ds\, dz \lesssim t^{-n} \iint_{\rho(B(y,t) \times (t/2,2t))} |\nabla u|^2 \, ds' \, dz'\\
 \lesssim \sup_{(y',t') \in \rho(B(y,t) \times (t/2,2t))} \fint_{z\in B(y',t')} \fint_{t'/2<s<2t'} |\nabla u|^2 \, ds\, dz,
\end{multline*}
where the last line holds because we can cover the image $\rho(B(y,t) \times (t/2,2t))$ by at most $N$ sets of the form $B(y',t') \times (t'/2,2t')$, with $N$ depending only on the bi-Lipschitz constants of $\rho$. Since $\rho(x) = x$ (and again $\rho$ is bi-Lipschitz), we also know that the image $\rho(B(y,t) \times (t/2,2t))$ is included in the cone $\Gamma_*(x):= \{(y',t') \in \R^n_+, \, |y'-x|<K_*t'\}$ with a large enough aperture $K_*$ that depends only on the bi-Lipschitz constants of $\rho$.  We deduce that 
\[\wt N(\nabla [u\circ \rho])(x) \lesssim \wt N_*(\nabla u)(x),\]
where $\wt N_*(\nabla u)(x)$ is defined with the cone $\Gamma_*(x)$ instead of $\Gamma(x)$, and hence  
\[\|\wt N(\nabla [u\circ \rho])\|_q \lesssim \|\wt N_*(\nabla u)\|_q.\]
The claim \eqref{claimlast} is now a consequence of $\|\wt N_*(\nabla u)\|_q \approx \|\wt N(\nabla u)\|_q$, which is a classical result in real analysis, see for instance Chapter II, equation (25) of \cite{Stein93}.
\end{proof}

The regularity problem is solvable for the Laplacian on $\R^{n}_+$ and Carleson perturbations preserve the solvability of the regularity problem in $L^q$ for small $q>1$, as shown in \cite[Theorem 2.1]{KP95}, so we have

\begin{theorem} \label{RpTh}
If $\C$ is a matrix function on $\R^{n}_+$ such that $\C \in CM_{sup}$ and the operator $L:=-\diver(I+\C)\nabla$ is uniformly elliptic - see \eqref{defellip} - then there exists a $q>1$ such that the regularity problem for $L$ is solvable in $L^q$.
\end{theorem}

We shall prove the following corollary.

\begin{corollary} \label{RpCor}
Let $b$ be a scalar function on $\R^n_+$ which satisfies  $C_b  \leq b \leq C_b$ and $|t\nabla b| \in CM_{sup}(M)$ for some positive constants $C_b$ and $M$. We construct the operator $L:=-\div \B \nabla$, where 
\[\B := \begin{pmatrix} b^{-1} I & 0 \\ 0 & b \end{pmatrix}.\]

Then there exists $q>1$ that depends only on $C_b$ and $M$ such that the regularity problem for $L$ is solvable in $L^q$.
\end{corollary}

The corollary is not optimized at all, but the result is still new when $|t\nabla b| \in CM_{sup}(M)$ for a large Carleson constant $M$. So our method enlarge the class of operators for which the solvability of the regularity problem in $L^q$ is known.

\medskip

\noindent{ \em Proof of the corollary: }
We cannot use Proposition \ref{Main2} directly here, but we can adapt the proof to fit our situation. Let us sketch it.

\medskip

We shall use \eqref{defrho12} with $v=0$, so we have 
\begin{equation} \label{defrhoh}
\rho(y,t) := (y,h(y,t) t), \quad Jac_{\rho} = \begin{pmatrix} I  & t\nabla_x h \\ 0 & h +t\partial_t h \end{pmatrix}.
\end{equation}
If $2|t\nabla h| \leq h$, the map $\rho$ is a change of variable. Moreover, if $\A$ has the form 
\[\A := \begin{pmatrix} a^{-1} I & 0 \\ 0 & a \end{pmatrix}\]
then $\A_\rho:= \det(Jac_\rho) (Jac_\rho)^{-T} (\mathcal A \circ \rho) \, (Jac_\rho)^{-1}$ can be decomposed as
\begin{equation} \label{defArho2}
\A_\rho = \underbrace{\begin{pmatrix} ha^{-1} & 0 \\ 0 & h^{-1}a \end{pmatrix}}_{\B_\rho} + \ \mathcal O(|t\nabla h| + |a\circ \rho - a|).
\end{equation}
We write $\C_\rho$ for $\A_\rho - \B_\rho$. 

\medskip

We know that $|t\nabla b| \in CM_{sup}(M)$. Therefore $\|t\nabla b\|_\infty \lesssim M$ and there exists $N$ (depending only on $M$ and $C_b$) such that $\|t\nabla (b^{1/N})\|_\infty \leq C_b/2$. We use $\rho$ with $h= b^{1/N}$, which is a change of variable, so the conjugate of the operator $L$ by $\rho$ is the operator $L_1=-\diver \A_1 \nabla$ where \eqref{defArho2},
\[\A_1 = \underbrace{\begin{pmatrix} b^{-(N-1)/N} & 0 \\ 0 & b^{(N-1)/N} \end{pmatrix}}_{\B_1} + \ \C_1,\]
with 
\[|\C_1| \lesssim |t\nabla b| + |b\circ \rho - b| \lesssim |t\nabla b| \in CM_{sup}.
\]
By iteration, we find that the conjugate of $L$ by $\rho^k$ is $L_k=-\diver \A_k \nabla$ where $\A_k$ can be decomposed as
\[\A_k = \underbrace{\begin{pmatrix} b^{-(N-k)/N} & 0 \\ 0 & b^{(N-k)/N} \end{pmatrix}}_{\B_k} + \ \C_k,\]
where $|\C_k| \in CM_{sup}$. Taking $k=N$ yields that $|\A_N -I| \in CM_{sup}$ so Theorem \ref{RpTh} entails that there exists $q>1$ such that the regularity problem for $L_N$ is solvable in $L^q$. Proposition \ref{Prrhostab} concludes the proof.
\ep

\subsection{Dirichlet problem in higher codimension.}

As we said in remark after Proposition \ref{Main2}, our method can be adapted to degenerate elliptic operators on $\R^n\setminus \R^d:= \{(x,t) \in \R^d \times \R^{n-d}, \, t\neq 0\}$, $d<n-1$. In this domain, the `degenerate' uniformly elliptic operators are the ones that can be written as $L = - \diver[ |t|^{d+1-n} \A \nabla ]$, where $\A$ satisfies the classical elliptic condition \eqref{defellip}. A rich elliptic theory was developed in \cite{DFMprelim} by Guy David, Svitlana Mayboroda, and the author, and provides for the aforementioned operators a notion of elliptic measure similar to \eqref{defug}, and then a notion of Dirichlet and regularity problem. The Carleson measure condition \eqref{defCM} has also to be adapted, but the adaptation is fairly straightforward. 

Without entering into much details, we know from \cite{FMZ} and \cite{FenGinfty} that the Dirichlet problem is solvable in $L^p$ for sufficiently large $p$ when $\A$ can be written as $\A = \B + \C$ where $\B$ can be written as a bloc matrix in the form
\begin{equation} \label{decompB}
\B := \begin{pmatrix} B_1 & B_2 \\ B_3 & bI  \end{pmatrix}
\end{equation}
with $B_1$ being a $d\times d$ matrix, $(B_2)^T$ and $B_3$ being $(n-d) \times d$ matrices, $b$ being a scalar function, $I$ being the identity matrix on $\R^{n-d}$, and 
\begin{equation} \label{condB1}
|\C| + |t||\nabla b| + |t\nabla_x B_3| + |t|^{n-d} |\diver_t [|t|^{d+1-n} B_3]| \in CM_{sup}.
\end{equation}
We recall that $\diver_t A$ is a vector containing the divergence of each column of $A$. 
One can think of operators in the form $L=-\diver[|t|^{d+1-n} \B \nabla]$ with $\B$ as above to be the higher codimension analogue of the Dahlberg-Kenig-Pipher operators. Indeed, the proof of the solvability of the Dirichlet problem for those operators rely on methods similar to the ones used by Kenig-Pipher, in particular the use of a  transformation similar to \eqref{BtoBD}. 

By using our method based on a change of variable, the condition $|t|^{n-d}|\diver_t [|t|^{d+1-n} B_3]|$ would be replaced by the much nicer $|t||\nabla_t B_3|$, leading to the following alternative:

\begin{theorem} \label{DpTh}
Let $L=-\diver [|t|^{d+1-n} \A \nabla]$ be an operator for which $\A$ satisfies \eqref{defellip}. Assume that $\A$ can be decomposed as $\A = \B + \C$ where 
\[\B = \begin{pmatrix} B_1 & B_2 \\ B_3 & bI  \end{pmatrix}\]
and
\begin{equation} \label{condB2}
|\C| + |t||\nabla b| + |t\nabla B_3|  \in CM_{sup}.
\end{equation}
Then there exists a large $p< +\infty$ such that the Dirichlet problem for $L$ is solvable in $L^p$.
\end{theorem}

\begin{remark}
Both \eqref{condB1} and \eqref{condB2} are conditions that can appear naturally. The condition \eqref{condB1} is the one that we would get by operators constructed from DKP operators as in \cite[(4.6)]{DFMKenig}, while the condition \eqref{condB2} can appear when we flatten the graph of Lipschitz function in higher codimension (see Remark 1.25 in \cite{FenGinfty}).
\end{remark}

\bp The proof of the above result is not much different from the proof of Proposition \ref{Main2}, but let us highlight the small changes. We use the maps $\rho$ defined as
\[\rho(y,t) := (y + t v(y,t) , h(y,t) t)\]
where $v(y,t)$ is a $(n-d) \times d$ matrix and $h$ is a scalar function (we see values in $\R^d$ and $\R^{n-d}$ as horizontal vectors). As before, the map $\rho$ is a change of variable if $|t||\nabla v| + |t||\nabla h|$ is small enough.
The conjugate operator $L_\rho$ is such that $L_\rho:= -\diver[|t|^{d+1-n} \A_\rho \nabla]$ where 
\begin{equation} \label{defArhoLD}
\A_{\rho} = h^{d+1-n} \det(Jac_\rho) (Jac_\rho)^{-T} (\mathcal A\circ \rho) \, (Jac_\rho)^{-1},
\end{equation}
the extra term $h$ being due to the weight $|t|^{d+1-n}$. But it is fine because it will cancel out with $\det(Jac_\rho) \approx h^{n-d}$ to give, similarly to the codimension 1 case, that 
\begin{equation} \label{ArhoY}
\begin{split}
\A_{\rho} = \underbrace{\begin{pmatrix} * & * \\ B_3 - b h^{-1}v & bh^{-1} I \end{pmatrix}}_{\B_{\rho}}  + \ \mathcal O(|\B\circ \rho - \B| + |\C \circ \rho| + |t||\nabla h| + |t||\nabla v|).
\end{split}\end{equation}
The rest of the proof does not change from Proposition \ref{Main2}, which gives the existence of a bi-Lipschitz change of variable $\rho'$ that fixes the boundary $\R^d$ and for which the matrix $\B_{\rho'}$ after the change of variable is in the form
\[\B_{\rho'} = \begin{pmatrix} * & * \\ 0 & I  \end{pmatrix},\]
and thus can be treated by the previous literature. Going back from the solvability of the Dirichlet problem for $L_{\rho'}$ to the one of $L$ is then similar to Proposition \ref{Prrhostab}.\ep

\begin{remark}
One may think that we can use a matrix $H$ of order $n-1$ - instead of the scalar $h$ - in the change of variable $\rho$ to try to reduce the bottom right corner of $\B$ to identity. But doing so would change the expression \eqref{defArhoLD} to
\[\A_{\rho} = \left(\frac{|Ht|}{|t|}\right)^{d+1-n} \det(Jac_\rho) (Jac_\rho)^{-T} (\mathcal A \circ \rho) \, (Jac_\rho)^{-1}\]
and $\left(\frac{|Ht|}{|t|}\right)^{d+1-n}$ is killed by $\det(Jac_\rho)$ only when $H$ has the form $H=hI$ - i.e. $H$ is the identity matrix times a scalar function\footnote{and in few other cases not worth mentioning because they are way more complicated.}.
Therefore, the bottom right corner of the Jacobian of the change of variable $\rho$ will always be close - up to Carleson measures - to a scalar function times identity, and can only cancel out terms in the same form. That is why, in Theorem \ref{DpTh}, we only consider a bottom right corner of $\B$ which is a scalar function times identity.
\end{remark}

\end{document}